\documentclass[leqno,11pt,a4paper]{article}
\usepackage[left=3.0cm, top=2.5cm,bottom=2.5cm]{geometry}
\usepackage{amsmath,amsthm,mathrsfs,calc,float,graphicx,newtxtext,newtxmath,bm,color}
\usepackage[british]{babel}

\newtheorem {theorem}{Theorem}[section]
\newtheorem {proposition}[theorem]{Proposition}
\newtheorem {lemma}[theorem]{Lemma}
\newtheorem {corollary}[theorem]{Corollary}
\newtheorem {remark}[theorem]{Remark}

{\theoremstyle{definition}

}

\newtheorem{innerthmrep}{Theorem}
\newenvironment{thmrep}[1]
  {\renewcommand\theinnerthmrep{#1}\innerthmrep}
  {\endinnerthmrep}

\newtheorem{innercorrep}{Corollary}
\newenvironment{correp}[1]
  {\renewcommand\theinnercorrep{#1}\innercorrep}
  {\endinnercorrep}

\def\ba{\begin{array}}
\def\ea{\end{array}}
\def\bea{\begin{eqnarray} \label}
\def\eea{\end{eqnarray}}
\def\be{\begin{equation} \label}
\def\ee{\end{equation}}
\def\bit{\begin{itemize}}
\def\eit{\end{itemize}}
\def\ben{\begin{enumerate}}
\def\een{\end{enumerate}}

\def\CC{\varmathbb{C}}
\def\EE{\varmathbb{E}}

\def\HH{\varmathbb{H}}

\def\NN{\varmathbb{N}}
\def\PP{\varmathbb{P}}
\def\QQ{\varmathbb{Q}}
\def\RR{\varmathbb{R}}

\def\l{\lambda}

\def\bW{\mathbf{W}}

\def\bZ{\mathbf{Z}}

\def\cF{\mathcal{F}}

\def\cH{\mathcal{H}}

\def\cL{\mathcal{L}}

\def\Po{\textup{Po}}

\newcommand{\set}[2]{\left\{ #1 : #2\right\}}
\newcommand{\tr}[1]{\mathrm{tr}\left( #1\right)}
\newcommand{\abs}[1]{\left| #1\right|}
\newcommand{\card}[1]{\left| #1\right|}
\newcommand{\len}[1]{\left|#1\right|} 
\usepackage{color}

\def\sys{\mathrm{sys}}

\begin{document}

\title{\bfseries Poisson approximation of\\ the length spectrum of random surfaces}

\author{Bram Petri\footnotemark[1]\,\, and Christoph Th\"ale\footnotemark[2]}

\date{}
\renewcommand{\thefootnote}{\fnsymbol{footnote}}
\footnotetext[1]{Max Planck Institute for Mathematics, Bonn, Germany. Email: brampetri@mpim-bonn.mpg.de}

\footnotetext[2]{
Faculty of Mathematics at Ruhr University Bochum, Germany. Email: christoph.thaele@rub.de}

\maketitle

\begin{abstract}
\noindent Multivariate Poisson approximation of the length spectrum of random surfaces is studied by means of the Chen-Stein method. This approach delivers simple and explicit error bounds in Poisson limit theorems.  They are used to prove that Poisson approximation applies to curves of length up to order $o(\log\log g)$ with $g$ being the genus of the surface.
\end{abstract}

{\small\tableofcontents}

\section{Introduction}

Brooks and Makover \cite{BrooksMakover} introduced a notion of random hyperbolic surfaces in order to make statements about the geometry of a `typical' closed hyperbolic surface. In essence, their model consists of randomly gluing an even number of ideal hyperbolic triangles together along their sides and then adding points in the resulting cusps, we will provide more details on this construction in Section \ref{sec:model} below. 

Among the surfaces obtained are all branched covers of the Riemann sphere, branched over at most three points. A classical result of Bely\v{\i} \cite{Belyi} states that as algebraic curves over $\CC$, these are exactly the curves that are defined over $\overline{\QQ}$, the algebraic closure of $\QQ$. Because $\overline{\QQ}$ is dense in $\CC$, this implies that this model generates a dense set in every moduli space of compact Riemann surfaces, which justifies the use of the term `typical'. 

Many questions about the geometry of a hyperbolic surface can be answered by studying the lengths of closed geodesics on this surface. For instance, there exist finite sets of curves on a surface so that the hyperbolic metric on that surface is fixed once the lengths of these curves are known. Another example is Selberg's famous trace formula, which implies that if one knows the length spectrum of a hyperbolic surface, that is, the multiset of the lengths of all closed geodesics on that surface, one also knows the spectrum of its Laplacian. We refer the reader to the monograph of Buser \cite{Buser} for more background material on hyperbolic surfaces and their length spectra.

In this paper we study the bottom part of the length spectrum of a random surface that arises from the Brooks-Makover construction. Given a positive real number $\ell$, the number of closed geodesics of length $\ell$ defines a random variable $Z_{\ell,N}:\Omega_N\to\NN$, where $\Omega_N$ denotes our underlying probability space of random surfaces glued out of $2N$ ideal triangles, $N\in\NN$. In \cite{Petri}, the first author proved that given a finite number of such random variables (i.e., take a finite set of positive real numbers and count the geodesics of exactly these lengths), they converge in distribution to independent Poisson random variables, as $N\to\infty$, with explicit parameters. This was proved using the classical method of moments. With this method, a Poisson limit theorem is derived from the convergence of all moments to those of a Poisson random variable. 

Our goal in the present paper is to provide a transparent and more flexible approach to this result that at the same time allows to extend its range of applicability. The latter extension was one of the main sources of motivation for us. The technical device we use is the so-called Chen-Stein method for Poisson approximation. One of the main advantages of this technique is that it allows one to deduce explicit error bounds on the multivariate total variational distance $d_{mTV}$ between the random variables and their limits. They usually deliver neat conditions on the interaction of the parameters involved under which a Poisson limit theorem holds. It is this precise knowledge that implies that, as opposed to fixed finite parts of the length spectrum, we are now able to handle moderately growing parts of the length spectrum as well. 

In order to properly state our results, which will be the content of Section \ref{sec:result}, we need to introduce some more notation. The lengths of geodesics on the surface before compactification can be computed using the matrices
\[L=\left(\begin{array}{cc} 1 & 1 \\ 0 & 1 \end{array} \right) \;\;\text{and}\;\; R=\left(\begin{array}{cc} 1 & 0 \\ 1 & 1 \end{array} \right).\]
This procedure goes as follows. We trace the geodesic $\gamma$ and when it passes from one triangle to another we record an `$L$' or an `$R$' whenever it turns left or right, respectively. Interpreting the resulting word $w_\gamma$ as a product of matrices, the hyperbolic length of $\gamma$ is given by
\[\ell(\gamma) = 2\cosh^{-1}\left(\frac{\tr{w_\gamma}}{2}\right) . \]
Results by Brooks \cite{Brooks} and Brooks-Makover \cite{BrooksMakover} guarantee that the lengths of geodesics on the compactified surface can be compared to those of their pre-images on the surface before compactification. See Section \ref{sec:geom} below for more details. 

The upshot of the above is that to understand the length spectrum of a random surface, we need to understand how often given words in $L$ and $R$ appear as cycles in the triangulation of these surfaces. Note that we may consider words up to cyclic permutation of their letters and an involution that consists of reading the word backwards and interchanging $L$ and $R$.

\begin{thmrep}{\ref{thm:Main}} Let $N\in\NN$ and let $W$ be a set of words in $L$ and $R$. Furthermore, set $m_W$ equal to the maximal length of a word in $W$ and assume that $m_W\leq N$. Then
$$
d_{mTV}(\bZ_{W,N},\bZ_W) \leq 18\,|W|^3\,(6m_W)^{3m_W+4}\,\frac{1}{N}\,,
$$
where $\bZ_{W,N}$ is the vector valued variable that counts the number geodesics of the types in $W$ and $\bZ_W$ is a vector of independent Poisson random variables whose means are determined by the words in $W$ (see Section \ref{sec:result} for details).
\end{thmrep}

Of course, this implies a Poisson limit theorem when the error is sufficiently small.

\begin{correp}{\ref{cor1}}
Suppose that $(W_N)_{N\in\NN}$ is a sequence of sets of combinatorial types of geodesics that satisfies
$$
\lim_{N\to\infty}\card{W_N}^3\,(6m_{W_N})^{3m_{W_N}+4}\,\frac{1}{N}=0\,.
$$
Then one has the Poisson limit theorem
$$
\lim_{N\to\infty}d_{mTV}\big(\bZ_{{W_N},N},\bZ_{W_N}\big)=0\,.
$$
\end{correp}

As a concrete example of a more geometric consequence we present the following result. Here and below, we write $a_N=o(b_N)$ for two sequences $(a_N)_{N\in\NN}$ and $(b_N)_{N\in\NN}$ provided that $a_N/b_N\to 0$ as $N\to\infty$.

\begin{correp}{\ref{cor2}}
The Poisson limit theorem holds for all curves on the surface with hyperbolic length up to $o(\log\log N)$.
\end{correp}

From the fact that the genus $g$ of a random surface is linearly bounded by $N$, we obtain:

\begin{correp}{\ref{cor3}}
The Poisson limit theorem holds for all curves on the surface with hyperbolic length up to $o(\log\log g)$.
\end{correp}

There is a parallel between the type of questions we ask here and questions that arise in the context of random regular graphs, that is, random graphs of which all the vertices have the same degree. In fact, the model we use for random surface naturally corresponds to the so-called configuration model for random $3$-regular (trivalent) graphs. Closed curves then correspond to cycles on graphs. In that context, the fact that the finite cycle counts are asymptotically Poisson distributed is a classical result due to Bollob\'as \cite{Bollobas}. Much more recently, McKay, Wormald and Wysocka  \cite{McKayWormaldWysocka} proved, using different methods than ours, that Poisson approximation holds for curves with a number of edges that is roughly logarithmic in the number of vertices. At first sight this seems a much larger range, but in fact it is only the translation from `combinatorial length' (the number of triangles that a geodesic goes through) to hyperbolic length that makes the extra logarithm appear in our setting. McKay, Wormald and Wysocka also use their result to give a probabilistic proof of the existence of regular graphs with large girth. Here, the girth of a graph is the length of the shortest cycle. Their proof provides graphs with girth that grows like a logarithm of the number of vertices. Up to a multiplicative constant, this is the fastest growth one can hope for. It is worth noting that there are also deterministic constructions for sequences of regular graphs with such girth, see for instance the work of Erd\H{o}s and Sachs \cite{ErdosSachs} and of Lubotzky, Phillips and Sarnak \cite{LubotzkyPhillipsSarnak}. In our set-up we can use our result to prove the existence of surfaces with moderately growing sytoles. The latter is the length of the shortest curve on a surface and is the analogue for surfaces of the girth of a graph. However, again because of the translation between hyperbolic and combinatorial length, we obtain surfaces with systoles that grow much more modestly than logarithmically in the genus. Because of the existence of sequences of hyperbolic surfaces with systoles that do in fact grow like a logarithm of the genus (see the work of Buser and Sarnak \cite{BuserSarnak}, Katz, Schaps and Vishne \cite{KatzSchapsVishne}, and Petri and Walker \cite{PetriWalker}), we will not pursue this direction any further in the current text.

It is worth mentioning that there are also other models of random surfaces which have been studied in the literature. For instance, Mirzakhani \cite{Mirzakhani} investigated a model for random surfaces related to Weil-Petersson geometry. Moreover, Guth, Parlier and Young \cite{GuthParlierYoung} were able to apply the probabilistic method to prove an existence result. They did this both in the case of hyperbolic surfaces, using the model coming from the Weil-Petersson metric, and surfaces that arise by randomly gluing together a finite number of equilateral Euclidean triangles.

The Chen-Stein method for Poisson approximation, the main tool in the proof of Theorem \ref{thm:Main}, has its roots in the work of Chen \cite{Chen}. It follows a general methodology introduced a few years earlier by Stein in the context of normal approximation, see \cite{Stein}. The method has become popular especially after the appearance of the influential paper of Arratia, Goldstein and Gordon \cite{ArratiaGoldsteinGordon} and the monograph of Barbour, Holst and Janson \cite{BarbourHolstJanson}, showing in particular that `two moments suffice for Poisson approximation'. More precisely, the main results in this context say that to show convergence in distribution of a sequence of random variables towards a Poisson random variable, one only has to control the behaviour of the first two moments. It should be emphasized that this constitutes a drastic simplification of the classical method of moments, for which one has to show that \textit{all} moments converge. Even more strikingly, the Chen-Stein method usually also has the advantage of providing a simple upper bound on the speed of convergence in the Poisson limit theorem that, as anticipated above, allows to deduce precise conditions on the interaction of the parameters involved under which such a statement applies. Having at hand such a quantitative result also to approximate individual probabilities effectively. The latter point was another source of motivation for us to apply the Chen-Stein method in the context of random surfaces.

Since the subject of this paper lies between hyperbolic geometry and probability and we would like to address readers from both communities, we have tried to write the text with both backgrounds in mind. For that reason, we have included some background material on geometry in Section \ref{sec:geom} and on probability theory in Section \ref{sec:prob}. Our main results are formally presented in Section \ref{sec:result}, which also contains their proofs.

\section{Geometric background: Random surfaces}\label{sec:geom}

The main goal of this paper will be to apply the Chen-Stein method, presented in the next section, in the context of random surfaces. The model for random surfaces that we will study in this text is the combinatorial model for `typical' hyperbolic surfaces introduced by Brooks and Makover in \cite{BrooksMakover}. The idea of this model is as follows. We start with an even number of ideal hyperbolic triangles and pair up the sides of these triangles randomly. Given such a pairing, we will create a hyperbolic surface by gluing the sides according to the pairing. In what follows we shall describe the details and the background of this construction.

\subsection{Hyperbolic geometry}

Before we continue to properly describe the probability space we first recall some basic facts about hyperbolic geometry and specifically the geometry of ideal hyperbolic triangles, for a more complete exposition on the geometry of hyperbolic surfaces we refer the reader to the monographs of Beardon \cite{Beardon} and Buser \cite{Buser}.

The model for the hyperbolic plane that we will use is the upper halfplane model. This means that the hyperbolic plane is the real $2$-manifold
\begin{equation*}
\HH^2 := \set{z\in\CC}{\Im(z)>0}\,,
\end{equation*}
with $\Im(z)$ standing for the imaginary part of $z\in\CC$, which is endowed with a Riemannian metric derived from the differential
\begin{equation*}
ds=\frac{|dz|}{\Im(z)}\,.
\end{equation*}
This is a metric of constant curvature $-1$ and the isometries of this metric are exactly the biholomorphisms of $\HH^2$ as a domain in $\CC$. It induces a distance function that we denote by $d_{\HH^2}(\,\cdot\,,\,\cdot\,)$ and is given by
$$
d_{\HH^2}(z,z')={\rm cosh}^{-1}\bigg(1+\frac{1}{2}\frac{|z-z'|}{\Im(z)\Im(z')}\bigg)
$$
for all $z,z'\in \HH^2$. The boundary of $\HH^2$ in this model is obtained by compactifying the real line with one point. As a set we will write $\partial\HH^2=\RR\cup\{\infty\}$.

The geodesics of $\HH^2$ are vertical lines and half circles orthogonal to the real line. This means that any geodesic naturally runs between two points of $\partial\HH^2$ and is uniquely defined by these two points (a vertical line is a geodesic between a point on the real line and the point at infinity).

\begin{figure}[t]
\begin{center} 
\includegraphics[scale=1.2]{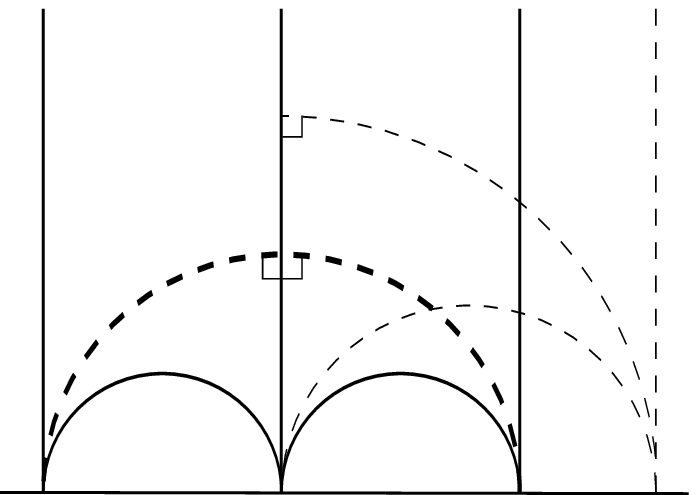} 
\caption{Shear.}
\label{fig:shear}
\end{center}
\end{figure}

There are many equivalent ways to define what a hyperbolic surface is. Perhaps the most natural definition is that a hyperbolic surface is a $2$-manifold that is locally modelled over the hyperbolic plane. It has charts with images in $\HH^2$ and transition maps that are restrictions of hyperbolic isometries.  The fact that the isometries of $\HH^2$ are biholomorphisms implies that a hyperbolic surface automatically comes with a complex structure (sometimes also called conformal or Riemann surface structure). The Poincar\'e-Koebe Uniforization Theorem states that this identification between hyperbolic and complex structures is one to one. Namely, if $S$ is a surface with negative Euler characteristic that is endowed with a complex structure, then there exists a unique hyperbolic structure on $S$ that is biholomorphically equivalent to the given complex structure.

A nice way of constructing hyperbolic surfaces, and this is the only way that we will construct them in this article, is by gluing together pieces of $\HH^2$. More concretely, we will take some finite and even number of triangles in $\HH^2$ and identify pieces of the boundaries of these triangles (their sides) by isometries. 

A triangle in $\HH^2$ is the convex hull of three points (the vertices of the triangle) in $\partial\HH^2\cup\HH^2$. The vertices that lie on $\partial\HH^2$ we will generally take out of the triangle, because the hyperbolic metric is not defined at such points. Such a triangle is uniquely defined by the angles at its vertices. That is, if two triangles have the same angles then we can find an isometry of $\HH^2$ that maps one to the other. We will be particularly interested in ideal triangles. These are triangles with all their vertices on $\partial\HH^2$. As such, they necessarily have sides of infinite length and all angles equal to $0$, which also means that there is only one ideal triangle up to isometry.

Given two triangles we can glue them together by identifying a pair of sides of the two triangles by an isometry. This means that if we want to do this, the sides need to have the same length. If this side length is finite then there is only one isometry between the two sides. However, we want to glue sides of ideal triangles together. Between two sides of ideal triangles there is an infinite number of isometries, which is illustrated by Figure \ref{fig:shear}. It shows gluings of an ideal triangle (the one on the left) to two `different' ideal triangles, the one with the full lines and the one with the dashed lines. As we noted before, these two triangles are isometric. This does however not imply that these two gluings are isometric. In fact, they are not and this can be seen from the following procedure. Given such a gluing of a pair of sides, we drop two perpendicular geodesics to the glued pair of sides from the vertices opposite to this side. Figure \ref{fig:shear} shows these geodesics. As can be seen in the figure, the first pair of perpendiculars meet in the sides that are glued together and the second pair does not. This immediately shows that the two gluings are not isometric and it also gives us a way to parametrize the gluings. We can measure the signed distance between the two points where the perpendiculars meet the side. This signed distance will be called the shear, which determines the gluing up to orientation. In this article we will always choose a gluing that realizes shear $0$. 

Given a gluing of ideal triangles that pairs up every side with another, we obtain a hyperbolic surface with punctures (called cusps, which come from the lack of vertices in ideal triangles) and no boundary. Figure \ref{fig:surface} shows an example. 

\begin{figure}[t]
\begin{center} 
\includegraphics[width=0.9\columnwidth]{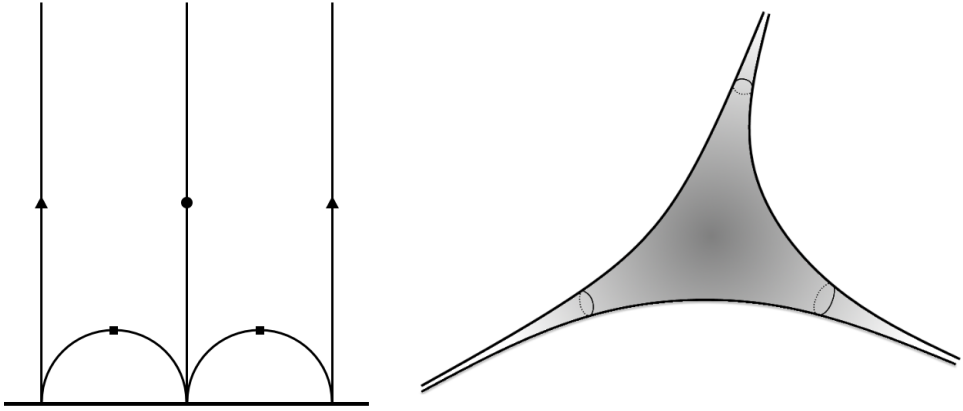} 
\caption{Example of a surface that arises from the two ideal hyperbolic triangles on the left by gluing together the sides marked with the same symbol (the common side marked by a circle is already identified). The resulting surface on the right has $3$ cusps and is homeomorphic to a $2$-sphere with three points removed.}
\label{fig:surface}
\end{center}
\end{figure}

In principle this surface could be unorientable. We will however only consider gluings such that the resulting surface is orientable. For every cusp, there exists a $t>0$ such that the cusp has a neighborhood isometric to the following open hyperbolic manifold:
\begin{equation*}
C_t := \set{z\in\HH^2}{\Im(z)>t}\Big/\sim\,,
\end{equation*}
where $\sim$ stands for the quotient by the hyperbolic isometry $z\mapsto z+1$. A horocycle around a cusp is the projection of a horizontal line in $C_t$.

The fact that the cupsed surface constructed above also comes with a complex structure gives us a way to turn it into a compact surface. The cusps of such a surface have neighborhoods that are biholomorphic to punctured disks in $\CC$. As such, the surface can be compactified by adding these points and by extending the complex structure. If the genus of the surface is at least $2$ (which implies negative Euler characteristic) then the Uniformization Theorem gives us a new hyperbolic structure. About the hyperbolic surfaces obtained through this procedure we have the following classical theorem due to Bely\v{\i}.

\begin{theorem}\label{thm:belyi}(\cite{Belyi}) The set of all hyperbolic surfaces obtained from the procedure described above is dense in any moduli space of closed hyperbolic surfaces.
\end{theorem}

\subsection{The model of Brooks and Makover} \label{sec:model}

We are now ready to describe the probability space of random surfaces we consider. We want to randomly glue together an even number ($2N$ with $N\in\NN$) of ideal triangles into a surface without boundary. As a set, our probability space is
\begin{equation*}
\Omega_N := \left\{\text{partitions of }\{1,\ldots,6N\}\text{ into pairs} \right\}
\end{equation*}
and since $\Omega_N$ is a finite set, a probability measure $\PP_N$ is obtained by normalizing the counting measure on $\Omega_N$. We shall write $\EE_N$ for the expectation with respect to $\PP_N$.

Given $\omega\in\Omega_N$ we obtain two surfaces: a surface with cusps and a compact surface. The surface with cusps is obtained by taking $2N$ ideal triangles and labelling their sides by the numbers $1,\ldots,6N$ in such a way that the numbers $3i-2,3i-1,3i$ form the sides of one triangle for all $i=1,\ldots, 2N$. We now glue these triangles together according to the pairings in $\omega$. In order to completely define the gluing we need to say in which way the side gluings are oriented. We will do this by gluing the triangles in such a way that the cyclic order of the labels on the sides defines an outward orientation on the resulting surface. The resulting cusped surface will be denoted $S_O(\omega)$. The compact surface corresponding to $\omega$ is obtained by compactifying $S_O(\omega)$ by the procedure described above and will be denoted $S_C(\omega)$.

Bely\v{\i}'s theorem (Theorem \ref{thm:belyi} above) implies that the set of surfaces $S_C(\omega)$ for all $\omega\in\bigcup_{N\in\NN}\Omega_N$ forms a dense set of surfaces in any moduli space of closed hyperbolic surfaces. Because of this, the model we described above gives us a way to say something about a `typical' compact hyperbolic surface, which was the motivation for Brooks and Makover to introduce this model in \cite{BrooksMakover}. The corresponding statement for the surfaces $S_O(\omega)$ is false. This can easily be seen from the facts that every pair of positive integers $(g,n)$ with $2g+n-2>0$ defines a $(6g+3n-6)$-dimensional moduli space of hyperbolic surfaces with genus $g$ and $n$ cusps and that we only obtain finitely many surfaces with genus $g$ and $n$ cusps of the form $S_O(\omega)$.

For this reason we are mainly interested in results on the geometry of the compact surfaces. Unfortunately, while the combinatorics of the gluing, encoded by $\omega$, give us the geometry of the surface $S_O(\omega)$, the dependence is not so clear for $S_C(\omega)$. The surface $S_C(\omega)$ is obtained by applying the Uniformization Theorem, which is an existence statement and hence does not tell us much about the geometry of the resulting surface. This situation can be remedied by applying the following comparison statement, proved by Brooks.

\begin{lemma}\label{lem:brooks}(\cite{Brooks}) For $L\in (0,\infty)$ sufficiently large, there is a constant $\delta(L)$ with the following property: Let $\omega\in\Omega_N$ such that the horocycles of length $L$ around the cusps of $S_O(\omega)$ are disjointly embedded. Then for every geodesic $\gamma$ in $S_C(\omega)$ there is a geodesic $\gamma'$ in $S_O(\omega)$ such that the image of $\gamma'$ is homotopic to $\gamma$, and
$$
\ell(\gamma) \leq \ell(\gamma') \leq (1+\delta(L))\ell(\gamma)\,,
$$
where $\ell(\gamma)$ and $\ell(\gamma')$ denotes the length of $\gamma$ and $\gamma'$, respectively. Furthermore, $\delta(L)\rightarrow 0$, as $L\rightarrow\infty$.
\end{lemma}

The statement above is actually a consequence of a much more general theorem by Brooks. Because we are mainly interested in curves, the lemma above is enough. The constant $L$ in the lemma is a measure of how far apart the cusps are. If $S_O(\omega)$ satisfies the condition of the lemma, we will say it has cusp length $\geq L$.

Lemma \ref{lem:brooks} implies that if we want to understand the lengths of curves on $S_C(\omega)$, we need to understand the lengths of curves on $S_O(\omega)$ and then prove that the cusps on $S_O(\omega)$ are far apart.  This last part was done by Brooks and Makover in \cite{BrooksMakover}. We will need a mild sharpening of their result, due to the first author.

\begin{proposition}\label{prp:cusplength}(\cite{Petri}) Let $L\in (0,\infty)$. We have
$$
\PP_N[S_O \text{ has cusp length}\ \geq\ L] = 1-\mathcal{O}(N^{-1})\,,
$$
as $N\rightarrow\infty$.
\end{proposition}

\subsection{The geometry of geodesics on $S_O(\omega)$}\label{sec:GeometryOfGeodesics}

To understand the geometry of curves on $S_O(\omega)$ it is convenient to use an alternative (but equivalent) definition of what a hyperbolic surface is. Namely, every orientable hyperbolic surface $S$ can be obtained as a quotient of the form
\begin{equation*}
S=\HH^2/G\,,
\end{equation*}
where
\begin{equation*}
G<\mathrm{Isom}^+(\HH^2) \simeq \mathrm{PSL}_2(\RR)
\end{equation*}
is a discrete and torsion-free subgroup. Here, ${\rm PSL}_2(\RR)$ denotes the projectivized group of all real $2\times 2$ matrices with determinant $1$. This representation is highly non-unique. For example, given any $h\in\mathrm{PSL}_2(\RR)$ the group $hGh^{-1}$ gives rise to an isometric copy of the surface $S$.

Given an element $g\in G$, its translation length is given by
\begin{equation*}
T_g := \inf_{x\in\HH^2}\{d_{\HH^2}(x,gx)\}\,.
\end{equation*}
There are two types of elements in $G$, namely, hyperbolic and parabolic elements, distinguished by the fact that $T_g>0$ when $g$ is hyperbolic and $T_g=0$ when $g$ is parabolic. For a hyperbolic element $g\in G$ one has that
\begin{equation*}
T_g=2\cosh^{-1}\left(\frac{\abs{\tr{g}}}{2}\right)\,,
\end{equation*}
where $\tr{g}$ denotes the trace of any representative of $g$ as an $\mathrm{SL}_2(\RR)$-matrix. Furthermore, the translation distance of a hyperbolic element $g$ is realized by a set of elements called the axis of $g$, which form a geodesic between the two fixed points of $g$.

It follows from the negative curvature of the metric on a hyperbolic surface that in every non-trivial homotopy class of closed curves there is a unique geodesic which minimizes the length among all curves in that homotopy class. Furthermore, we can find an element $g\in G$ such that the axis of $g$ projects to this curve and the length of the curve is equal to $T_g$.

We want to understand the lengths of geodesics on $S_O(\omega)$. The construction we presented only gives us the surface and not the corresponding group $G$, which means that there are many equivalent choices for such a group. We will however not reconstruct the entire group but rather work locally. Given a closed curve $\gamma$ on $S_O(\omega)$ we will find an element $g\in\mathrm{PSL}_2(\RR)$ that lies in some group uniformizing $S_O(\omega)$ and the axis of which under this identification projects to $\gamma$.

What we do is the following. Given a geodesic $\gamma$, we pick a point on $\gamma$ and a direction to travel in. This point lies on a certain ideal triangle, which we identify with the triangle with vertices $0$, $1$ and $\infty$ in $\HH^2$. Each time $\gamma$ goes into a new triangle, the orientation of the surface tells us whether this new triangle lies to the right or the left of the previous triangle. We trace $\gamma$ and copy the triangles into $\HH^2$ (with shear $0$ either to the right or the left of the previous triangle) until we reach the point from which we departed again. There is a unique orientation preserving isometry of $\HH^2$ that sends the first triangle we drew to the last one. This is the isometry we are looking for.

Because our gluing (of ideal triangles with shear $0$) is very particular, we can actually write down the isometry explicitly. This goes as follows. From the curve $\gamma$ and the orientation on the surface we obtain a word in $L$ (for left) and $R$ (for right) by tracing $\gamma$. Call this word $w_\gamma$. If we replace the letters $L$ and $R$ by the matrices
\begin{equation*}
L:=\left(\begin{array}{cc} 1 & 1 \\ 0 & 1 \end{array}\right)\quad \text{ and }\quad R:=\left(\begin{array}{cc} 1 & 0 \\ 1 & 1 \end{array}\right)
\end{equation*}
then $w_\gamma$ becomes a matrix. The length of $\gamma$ is given by
\begin{equation}\label{eq:LengthGamma}
\ell(\gamma) = 2\cosh^{-1}\left(\frac{\abs{\tr{w_\gamma}}}{2}\right)\,.
\end{equation}
Note that the identification $\gamma\mapsto w_\gamma$ is not a well-defined map. If we start tracing $\gamma$ at a different point then we obtain a cyclic permutation of the letters and if we trace $\gamma$ in the opposite direction we obtain the same word read backwards with $L$ and $R$ interchanged, which on the matrix level corresponds to taking the transpose. Luckily the trace of $w_\gamma$ is preserved under all these operations, which means that \eqref{eq:LengthGamma} does in fact make sense. 

From the discussion above we conclude that if we want to understand the length spectrum of $S_O(\gamma)$ then we need to find the words that appear as (homotopy classes of) curves on $S_O(\omega)$. Of course we only care about words up to cyclic permutation and the inversion described above. Because of this it will be convenient to use these operations to define an equivalence relation `$\sim$' on the set of words in $L$ and $R$. This set of words itself will be denoted by $\{L,R\}^*$ and we will denote the class of a word $w$ by $[w]$.

Now define the random variable
\begin{equation}\label{eq:DefZwN}
Z_{[w],N}:\Omega_N\to\mathbb{N}\,,
\end{equation}
which counts the number of homotopy classes (of geodesics) corresponding to $[w]$. Most of this paper will be concerned with studying the behavior of these random variables.

\subsection{Overview of known results}

Now that we have described the model properly, let us collect some of the properties that are known for these surfaces.

The first question that needs to be answered is what the topology of these surfaces is like. First of all, the surfaces might not even be connected, $\Omega_N$ for instance includes gluings of the surfaces into $N$ tori or spheres. These disconnected surfaces turn out to form a negligible set, which is a direct consequence of the connection between random surfaces and the configuration model for random trivalent graphs. Concretely, we have the following theorem, independently proved by Bollob\'as and Wormald.

\begin{theorem} (\cite{Bollobas,Wormald}) We have
$$
\lim_{N\rightarrow\infty} \PP_N[S_O\text{ and }S_C\text{ are connected}] = 1\,.
$$
\end{theorem}

Given the fact that $S_O$ and $S_C$ are asymptotically almost surely connected and come with a triangulation, we can compute their genus using the Euler characteristic. In case of a connected surface, we have
\begin{equation*}
g = 1+\frac{N}{2} - \frac{n}{2}\,,
\end{equation*}
where $n$ is the number of cusps of $S_O$ or equivalently the number of vertices of the given triangulation. This means that the random variable we need to understand is $n$. It turns out that on average $n$ is at most logarithmic in $N$. This can be proved by counting methods, which has been done in \cite{BrooksMakover,DunfieldThurston,GamburdMakover,PippengerSchleich}. By exploiting a connection to random permutations Gamburd \cite{Gamburd} was able to determine the precise asymptotic distribution of $n$. Because this is just an overview, we will content ourselves with the following weaker version of the result. We shall write $a_N\sim b_N$ for two real sequences $(a_N)_{N\in\NN}$ and $(b_N)_{N\in\NN}$ if $a_N/b_N\to 1$, as $N\to\infty$. Recall that we use the symbol $\EE_N$ to denote expectation with respect to the probability measure $\PP_N$.

\begin{theorem} (\cite{Gamburd}) We have
$$
\EE_N[g] \sim \frac{N}{2}\,,
$$
as $N\to\infty$.
\end{theorem}

The next question is how the geometry of these surfaces behaves. We will not go through all the known results. Instead, we give one sample result, in the form of the following theorem due to Brooks and Makover. To state it, let us recall the definition of the Cheeger constant and the systole of a compact surface $S$. The Cheeger constant $h(S)$ is
$$
h(S) := \inf_A \bigg\{\frac{\ell(\partial A)}{{\rm area}(A)}:{\rm area}(A)\leq \frac{1}{2}{\rm area}(S)\bigg\}\,,
$$
where the infimum runs over all open subsets $A\subseteq S$ with smooth boundary and ${\rm area}(\,\cdot\,)$ denotes the Riemannian volume (area) on $S$, while $\ell(\partial A)$ is the Riemannian length of the boundary $\partial A$ of $A$. Moreover, the systole $\sys(S)$ of $S$ is the length of the smallest non-contractible curve (loop) in $S$ and ${\rm diam}(S)$ denotes the diameter of $S$, that is, the maximal Riemannian distance that two points on $S$ can have. Further recall that we write $g(S)$ for the genus of $S$.

\begin{theorem}\label{thm:brooksmakover}(\cite{BrooksMakover}) There exist constants $C_1,\ldots,C_4\in(0,\infty)$ such that the following assertions are true.
\begin{itemize}
\item[(a)] The first non-trivial eigenvalue $\lambda_1$ of the Laplacian satisfies $$\lim_{N\to\infty}\PP_N [\lambda_1(S_C)>C_1]=1\,.$$
\item[(b)] The Cheeger constant $h$ satisfies $$\lim_{N\to\infty}\PP_N[h(S_C)>C_2]=1\,.$$
\item[(c)] The systole $\sys$ satisfies $$\lim_{N\to\infty}\PP_N [\sys(S_C)>C_3] = 1\,.$$
\item[(d)] The diameter ${\rm diam}$ satisfies $$\lim_{N\to\infty}\PP_N [\mathrm{diam}(S_C)<C_4\log\left(g(S_C)\right)] = 1\,.$$
\end{itemize}
\end{theorem}

Finally, we mention that Gamburd's topological results also have geometric implications, detailed in \cite{Gamburd}.

\section{Probabilistic background: The Chen-Stein method}\label{sec:prob}

\subsection{Characterization and Chen-Stein equation}

This section is devoted to an introduction to the Chen-Stein method for Poisson approximation. We have tried to highlight the key ideas and, like in the previous section, we also decided to include some material that by now is classical in the literature in order to keep the paper self-contained and to introduce the non-specialized reader to the subject.

Before providing the details, we use the occasion to introduce some notation. We let $(\Omega,\cF,\PP)$ be a probability space that is rich enough to carry all our (in all cases non-negative and integer-valued) random variables. If $Z$ is such a random variable, we write $\cL(Z)$ for its law, i.e., the image measure of $\PP$ on $\NN_0:=\{0,1,2,\ldots\}$ that is induced by $Z$. We say that $\cL(Z)$ is a Poisson distribution with parameter $\l\in(0,\infty)$, formally $\cL(Z)=\Po_\l$, if
$$
\PP[Z=k]={\l^k\over k!}e^{-\l}\,,\qquad k\in\NN_0\,.
$$
Expectation with respect to $\PP$ is denoted by $\EE$, so that in particular $\EE[Z]=\l$ if $\cL(Z)=\Po_\l$.

\medbreak

The starting point of the Chen-Stein method is the following characterization of Poisson random variables that goes back (at least) to the work of Chen \cite{Chen}. 

\begin{lemma}(\cite[Chapter XVIII.1]{Venkatesh})\label{lem:CharacterizationPoisson}
Let $Z$ be a non-negative integer-valued random variable. Then $\cL(Z)=\Po_\l$ for some $\l\in(0,\infty)$ if and only if
$$
\EE[\l g(Z+1)-Zg(Z)]=0
$$
for all bounded functions $g:\NN_0\to\RR$.
\end{lemma}

Bearing in mind Lemma \ref{lem:CharacterizationPoisson}, the basic philosophy of the Chen-Stein method can roughly be summarized as follows. If for an arbitrary non-negative integer-valued random variable $W$ on $(\Omega,\cF,\PP)$ the difference $\EE[\lambda g(W+1)-Wg(W)]$ is `close' to zero for all bounded functions $g:\NN_0\to\RR$ then $\cL(W)$ should be `close' to $\Po_\l$.

To measure the closeness of the distributions $\cL(X)$ and $\cL(Y)$ of two non-negative integer-valued random variables $X$ and $Y$ one can use different probability metrics. Let $\cH$ be a class of bounded test functions $h:\NN_0\to\RR$ and put
$$
d_\cH(X,Y) = d_\cH(\cL(X),\cL(Y)) := \sup_{h\in\cH}\big|\EE[h(X)]-\EE[h(Y)]\big|\,.
$$
If $\cH={\rm Ind}:=\{{\bf 1}_B:B\subseteq\NN_0\}$ is the class of indicator functions of subsets of $\NN_0$, then $d_\cH$ is nothing else than the total variation metric and if $\cH={\rm Lip}:=\{h:\NN_0\to\RR:|h(m)-h(n)|\leq|m-n|,m,n\in\NN_0\}$ is the class of Lipschitz functions on $\NN_0$ with Lipschitz constant bounded by $1$, then $d_\cH$ becomes the Wasserstein metric on the space of probability measures on $\NN_0$. We shall concentrate on the total variation metric in this paper and write $d_{TV}:=d_{\rm Ind}$, that is,
$$
d_{TV}(X,Y)=\sup_{B\subseteq\NN_0}\big|\PP[X\in B]-\PP[Y\in B]\big|\,.
$$ 
We emphasize that convergence in total variation distance implies convergence in distribution. More precisely, if $d_{TV}(X_N,X)\to 0$, as $N\to\infty$, for a sequence $(X_N)_{n\in\NN}$ of random variables and another random variable $X$, then $X_N$ converges in distribution to $X$, or, equivalently, $\cL(X_N)$ weakly converges to $\cL(X)$, as $N\to\infty$.

The next step is to connect the definition of the metric $d_{TV}$ to the characterization of Poisson distributions presented in Lemma \ref{lem:CharacterizationPoisson}. To this end, let $h\in{\rm Ind}$, that is, $h={\bf 1}_B$ for some subset $B\subset\NN_0$. The so-called Chen-Stein equation associated with $h$ is then given by
\begin{equation}\label{eq:ChenSteinEq}
h(k) - \EE[h(Z)]= \l g(k+1) - kg(k)\,,\qquad k\in\NN_0\,,
\end{equation}
where here and below $\cL(Z)=\Po_\l$. For a given $h\in{\rm Ind}$, the (unique) bounded solution of \eqref{eq:ChenSteinEq} taking value zero at zero is denoted by $g_h$ and we use the notation ${\rm Sol}$ to indicate the space of all solutions of \eqref{eq:ChenSteinEq} that arise in this way. The next major step is to gain control on the properties of the solutions of the Chen-Stein equation \eqref{eq:ChenSteinEq}. Most importantly, one needs bounds on $g_h$ and its first-order forward difference, that is, bounds on
$$
\|g_h\|_\infty:=\sup_{k\in\NN_0}|g_h(k)|\qquad\text{and}\qquad\|\Delta g_h\|_\infty:=\sup_{k\in\NN_0}|\Delta g_h(k)|\,,
$$
where $\Delta g_h(k)=g_h(k+1)-g_h(k)$.

\begin{lemma}(\cite[Chapter XVII.3 and Theorem XVIII.3.2]{Venkatesh})\label{lem:Solutions}
If $h\in {\rm Ind}$, then
$$\|g_h\|_\infty\leq 3\min\Big\{1,{1\over\sqrt{\l}}\Big\}\qquad\text{and}\qquad\|\Delta g_h\|_\infty\leq{1-e^{-\l}\over\l}\leq\min\Big\{1,{1\over\l}\Big\}\,,
$$
independently of $h$. 
\end{lemma}

Now, let $W$ be a non-negative integer-valued random variable on $(\Omega,\cF,\PP)$ that we wish to compare to the Poisson variable $Z$. We plugin $W$ for $k$ in \eqref{eq:ChenSteinEq}, take expectations and absolute values on both sides of the Chen-Stein equation and finally the supremum over all $h\in{\rm Ind}$. This gives
\begin{equation}\label{eq:ChenSteinStart}
d_{TV}(W,Z) = \sup_{g\in{\rm Sol}}\big|\EE[\l g(W+1)-Wg(W)]\big|\,.
\end{equation}
It turns out that the right-hand side (RHS) of \eqref{eq:ChenSteinStart} can further be bounded in many situations, cf.\ the monograph \cite{BarbourHolstJanson} for a collection of representative examples as well as the more recent survey article \cite{ChenRoellin}. We emphasize at this point that the RHS of \eqref{eq:ChenSteinStart} does no more involve the Poisson random variable $Z$. The fact that we are comparing $\cL(W)$ with ${\rm Po}_\lambda$ is encoded by the special form of the RHS that reflects the characterization of Poisson random variables that we formulated in Lemma \ref{lem:CharacterizationPoisson}.

\subsection{Poisson limit theorems}

We consider the following set-up that is taken from Chapter XVIII.7 in \cite{Venkatesh}, which in turn follows the presentation in \cite{ArratiaGoldsteinGordon}. Let $\Gamma\neq\emptyset$ be a finite index set and $\Xi_\Gamma:=\{X_\alpha:\alpha\in\Gamma\}$ be a family of random variables on a probability space $\Omega$ that only take values in $\{0,1\}$. It is neither required that these random variables are identically distributed nor that they are independent. In fact, to capture the dependence structure of $\Xi_\Gamma$, we define, for each $\alpha\in\Gamma$
\begin{itemize}
\item[-] $\Gamma_\alpha'$, a collection of indices $\beta$ of those random variables $X_\beta\in\Xi_\Gamma$ with $\alpha\neq\beta$ that are thought of as strongly dependent on $X_\alpha$ and
\item[-] $\Gamma_\alpha:=\Gamma\setminus(\Gamma_\alpha'\cup\{\alpha\})$, a collection of indices of random variables from $\Xi_\Gamma$ that are thought of as weakly dependent on $X_\alpha$.
\end{itemize}
so that for each $\alpha\in\Gamma$ the index set $\Gamma$ has a decomposition as the disjoint union $\Gamma=\{\alpha\}\sqcup\Gamma_\alpha'\sqcup\Gamma_{\alpha}$ and one has that $\Gamma_\alpha'=\emptyset$ and $\Gamma_\alpha=\Gamma\setminus\{\alpha\}$ if $\Xi_\Gamma$ is a family of mutually independent random variables. 

We furthermore assume that for each $\alpha\in\Gamma$, we are given a collection random variables $\{X_{\alpha,\beta}':\beta\in\Gamma \setminus \{\alpha\}\}$ that satisfy
\[X_{\alpha,\beta}'(\omega)\geq X_\beta(\omega)\;\text{for all}\;\beta\in\Gamma_\alpha\,,\omega\in\Omega\]
and
\[\PP[X_{\alpha,\beta}'=1] \;=\;\PP[X_{\beta}'=1|\;X_\alpha = 1]\;\text{for all}\;\beta\in\Gamma\setminus\{\alpha\}\]

(this is known as a monotone coupling).

It turns out that the distance between the distribution of the random variable $W:=\sum\limits_{\alpha\in\Gamma}X_\alpha$ and a suitable Poisson distribution can be assessed by means of the four quantities
\begin{align*}
 \Sigma_1 & := \sum_{\alpha\in\Gamma}(\EE[X_\alpha])^2\,, & \Sigma_2: & =  \sum_{\alpha\in\Gamma}\sum_{\beta\in\Gamma_\alpha'}(\EE[ X_\alpha])(\EE[ X_\beta])\,, \\
  \Sigma_3 & :=  \sum_{\alpha\in\Gamma}\sum_{\beta\in\Gamma_\alpha'}\EE[X_\alpha X_\beta]\,, &  \Sigma_4 & :=  \sum_{\alpha\in\Gamma}\sum_{\beta\in\Gamma_\alpha}\EE[X_\alpha X_\beta]-(\EE[X_\alpha])(\EE[X_\beta])\,. 
\end{align*}
We emphasize that the terms $\Sigma_1$, $\Sigma_2$, $\Sigma_3$ and $\Sigma_4$ only contain information on the first- and second-order moments as well as on the dependency neighborhoods $\Gamma_\alpha'$ of the random variables $X_\alpha\in \Xi_\Gamma$. Moreover we note that these quantities all only depend on the random variables $X_\alpha$ and not on the auxiliary variables $X_{\alpha,\beta}'$.

We can now rephrase the following simplified version of the result from Chapter XVIII.7 in \cite{Venkatesh}. It is known as the `coupling' approach to Poisson approximation via the Chen-Stein method. To keep the paper self-contained and to illustrate how the Chen-Stein method works in practice, we have decided to include the short proof.

\begin{theorem}\label{thm:Stein1D}
Define the random variable $W:=\sum\limits_{\alpha\in\Gamma}X_\alpha$, put $\lambda:=\sum\limits_{\alpha\in\Gamma}\EE[X_\alpha]$ and let $Z$ be a random variable such that $\cL(Z)=\Po_\lambda$. Suppose that $\lambda\in(0,\infty)$. Then
$$
d_{TV}(W,Z)\leq \min\Big\{1,{1\over\l}\Big\}\,\big(\Sigma_1+\Sigma_2+\Sigma_3+\Sigma_4\big)\,.
$$
\end{theorem}
\begin{proof}
Set $U_\alpha = W$ and $V_\alpha = \sum_{\beta\in\Gamma_\alpha\setminus\{\alpha\}}X_{\alpha,\beta}'$. With equation \eqref{eq:ChenSteinStart} in mind, note that
\[\EE[X_\alpha g(W)] = \EE[X_\alpha]\cdot\EE\left[\left. g\left(1+\sum_{\beta\in\Gamma\setminus\{\alpha\}}X_\beta\right)\; \right| X_\alpha =1\right] = \EE[X_\alpha]\cdot\EE[g(V_\alpha +1)]. \]
This implies that
\[\EE[\lambda g(W+1)- Wg(W)] = \sum_{\alpha\in\Gamma}\EE[X_\alpha]\cdot\EE[g(U_\alpha)-g(V_\alpha+1)]. \]
We have
\[ \abs{\EE[g(U_\alpha)-g(V_\alpha+1)]} \leq ||\Delta g||_\infty \cdot \EE[\abs{U_\alpha-V_\alpha}] \]
and thus obtain that
\[\abs{\EE[\lambda g(W+1)- Wg(W)]} \leq ||\Delta g||_\infty \sum_{\alpha\in\Gamma}\EE[X_\alpha]\cdot\EE[\abs{U_\alpha-V_\alpha}]. \]
Let us analyze this sum above term by term. By definition, we get that
\begin{eqnarray*}
\abs{U_\alpha-V_\alpha} & = & \abs{X_\alpha - \sum_{\beta\in\Gamma\setminus\{\alpha\}} X_{\alpha,\beta}'-X_\beta } \\[3mm]
 & \leq & X_\alpha + \sum_{\beta\in \Gamma_\alpha'}X_{\alpha,\beta}'+X_\beta + \sum_{\beta\in \Gamma_\alpha} X_{\alpha,\beta}' - X_\beta\,.
\end{eqnarray*}
Hence, using that $\EE[X_\alpha]\cdot\EE[X_{\alpha,\beta}'] = \EE[X_\alpha X_\beta]$, we see that
\begin{eqnarray}\label{eq:ProofCS3}
\abs{\EE[\lambda g(W+1)- Wg(W)]} \leq ||\Delta g||_\infty \sum_{\alpha\in\Gamma}\left(\EE[X_\alpha]^2 + \sum_{\beta\in\Gamma_\alpha '} \EE[X_\alpha X_\beta] + \EE[X_\alpha]\cdot \EE[X_\beta]\right. \notag \\
 + \left. \sum_{\beta\in\Gamma_\alpha}\EE[X_\alpha X_\beta] - \EE[X_\alpha]\cdot \EE[X_\beta]\right). 
\end{eqnarray}
A glance at \eqref{eq:ChenSteinStart} shows that this implies the result after we take the supremum over all $g\in{\rm Sol}$ in \eqref{eq:ProofCS3} and using Lemma \ref{lem:Solutions}.
\end{proof}

Note that once having derived a bound on the total variation distance $d_{TV}(W,Z)$ as above, this readily allows one to approximate individual probabilities with an explicit error bound. For example, if $W$ and $Z$ are random variables as in Theorem \ref{thm:Stein1D} we see that
$$
\big|\PP[W=k]-\PP[Z=k]\big|\leq 2d_{TV}(W,Z)
$$
for all $k\in\NN$. In particular, one has that
$$
\big|\PP[W=0]-e^{-\l}\big|\leq 2d_{TV}(W,Z)\,.
$$
A similar comment also applies to the multivariate setting that we are going to describe next.

\begin{remark}\rm 
For many purposes -- and especially the one of the present paper -- it is sufficient to use Theorem \ref{thm:Stein1D} in a form in which the constant $\min\{1,1/\l\}$ there is replaced by $1$.
\end{remark}

We shall later apply a multivariate version of Theorem \ref{thm:Stein1D}. For this we adapt the set-up introduced above, let $d\in\NN$ and $\{\Gamma_1,\ldots,\Gamma_d\}$ be a partition of the index set $\Gamma$ into non-empty disjoint subsets. Our interest is in the random vector
$$
\bW=(W_1,\ldots,W_d)\qquad\text{with}\qquad W_i:=\sum_{\alpha\in\Gamma_i}X_\alpha\,,\quad i\in\{1,\ldots,d\}\,,
$$
that we would like to compare with a random vector $\bZ=(Z_1,\ldots,Z_d)$ consisting of independent random variables $Z_i$ with $\cL(Z_i)=\Po_{\lambda_i}$, where $\lambda_i=\sum\limits_{\alpha\in\Gamma_i}\EE[X_\alpha]$ for $i\in\{1,\ldots,d\}$. To this end, we use the multivariate total variation distance
$$
d_{mTV}(\bW,\bZ):=\sup_{B\subseteq\NN_0^d}\big|\PP[\bW\in B]-\PP[\bZ\in B]\big|
$$
and recall that, like in the univariate case, convergence in the multivariate total variation distance implies convergence in distribution of the involved random variables (or weak convergence of their laws).
To present a bound for $d_{mTV}(\bW,\bZ)$, we introduce for $i\in\{1,\ldots,d\}$ the quantities
\begin{align*}
\Sigma_{1,i} & :=\sum_{\alpha\in\Gamma_i}(\EE[X_\alpha])^2\,, & \Sigma_{2,i} & :=\sum_{\alpha\in\Gamma_i}\sum_{\beta\in\Gamma_\alpha'}(\EE[ X_\alpha])(\EE[ X_\beta]), \\ \Sigma_{3,i} & :=\sum_{\alpha\in\Gamma_i}\sum_{\beta\in\Gamma_\alpha'}\EE[X_\alpha X_\beta], & 
\Sigma_{4,i} & :=\sum_{\alpha\in\Gamma_i}\sum_{\beta\in\Gamma_\alpha}\EE[X_\alpha X_\beta]-(\EE[ X_\alpha])(\EE[ X_\beta]),
\end{align*}
that are defined similarly as $\Sigma_1$, $\Sigma_2$, $\Sigma_3$ and $\Sigma_4$ above. We stress that we again assume the existence of auxiliary random variables $X_{\alpha,\beta}'$.

\begin{theorem}\label{thm:SteinMulti}
Suppose that $\l_i\in(0,\infty)$ for all $i\in\{1,\ldots,d\}$. Then
$$
d_{mTV}(\bW,\bZ)\leq 3\sum_{i=1}^d\big(\Sigma_{1,i}+\Sigma_{2,i}+\Sigma_{3,i}+\Sigma_{4,i}\big)
$$
\end{theorem}

The proof of Theorem \ref{thm:SteinMulti} follows by suitably modifying that of Theorem \ref{thm:Stein1D} and for this reason we have decided not to include the details. The basic idea is to successively replace the components of the vector $\bW$ by those of $\bZ$, one after the other. In each step, the error is controlled similarly as in the proof of Theorem \ref{thm:Stein1D}, but this time using the bound $\|g_h\|\leq 3\min\{1,1/\sqrt{\l}\}\leq 3$ from Lemma \ref{lem:Solutions}, whence the factor $3$ in Theorem \ref{thm:SteinMulti}. This finally leads to the desired result, see \cite{ArratiaGoldsteinGordon,BarbourHolstJanson} for further information.

\section{The length spectrum}\label{sec:result}

\subsection{The main result}

In what follows we shall use the notation and the quantities defined in Section \ref{sec:GeometryOfGeodesics}. In \cite{Petri} it was proved that for a finite set
\begin{equation*}
W\subset \{L,R\}^* /\sim
\end{equation*}
the random variables $Z_{[w],N},[w]\in W$ defined at \eqref{eq:DefZwN} are asymptotically independently Poisson distributed with means $\lambda_{[w]}$, where
\begin{equation*}
\lambda_{[w]} = \frac{\card{[w]}}{2\len{w}}\,.
\end{equation*}
Here, $\len{w}$ denotes the number of letters in the word $w$ and $\card{[w]}$ the cardinality of the set $[w]$. This was proved using the classical method of moments.

In what follows we will provide an alternative and transparent approach as well as a sharpening of this result. That is, we will apply the Chen-Stein method described in Section \ref{sec:prob}, which will give us the approximation with error bounds. Before we state the theorem, we define the vector valued random variables we are interested in. Given $W\subset \{L,R\}^*/\sim$, we define the vector
\begin{equation*}
\bZ_{W,N} := (Z_{[w],N})_{[w]\in W}: \Omega_N\rightarrow\NN^W\,.
\end{equation*}
Furthermore, we define $\bZ_W$ to be a vector
\begin{equation*}
\bZ_W := (Z_{[w]})_{[w]\in W}\,,
\end{equation*}
where the coefficients $Z_{[w]}$ are independent random variables with law  $\cL(Z_{[w]})=\Po_{\lambda_{[w]}}$.

\begin{theorem}\label{thm:Main} Let $N\in\NN$ and $W\subset \{L,R\}^* /\sim$, put $m_W:=\max\set{\len{w}}{[w]\in W}$ and assume that $m_W\leq N$. Then
$$
d_{mTV}(\bZ_{W,N},\bZ_W) \leq 18\,\card{W}^3\,(6m_W)^{3m_W+4}\,{1\over N}\,.
$$
\end{theorem}

\begin{remark}\rm  
\begin{itemize}
\item[(a)] If we are interested in $1$-dimensional Poisson approximation only, that is, in a bound on $d_{TV}(Z_{[w],N},Z_{[w]})$ for some fixed $[w]\in\{L,R\}^*/\sim$, we can slightly improve the constant in Theorem \ref{thm:Main}. In fact, using Theorem \ref{thm:Stein1D} instead of Theorem \ref{thm:SteinMulti} in the proof given below one readily sees that the factor $18$ can then be replaced by the factor $6$.
\item[(b)] Note that Lemma \ref{lem:brooks} and Proposition \ref{prp:cusplength} imply that up to an error of the order $\mathcal{O}(N^{-1})$ the length spectrum of the compactified random surfaces is also determined by the random variables $Z_{W,N}$. 
\end{itemize}
\end{remark}

Before we prove our theorem in Section \ref{subsec:Proofs}, we turn to some consequences. The first one is a Poisson limit theorem, that, in particular, provides a concrete condition under which the bound on the multivariate total variation distance in Theorem \ref{thm:Main} tends to zero. As explained in the previous section, such a quantitative Poisson limit theorem allows one to approximate individual probabilities like by those of a vector of independent Poisson random variables with an explicit upper bound on the error.

\begin{corollary}\label{cor1}
Suppose that $(W_N)_{N\in\NN}$ is a sequence with $W_N\subset\{L,R\}^{*}/\sim$ for any $N\in\NN$ that satisfies $m_{W_N}\leq N$ for all $N\in\NN$ and
$$
\lim_{N\to\infty}\card{W_N}^3\,(6m_{W_N})^{3m_{W_N}+4}\,{1\over N}=0\,.
$$
Then one has the Poisson limit theorem
$$
\lim_{N\to\infty}d_{mTV}\big(\bZ_{{W_N},N},\bZ_{W_N}\big)=0\,.
$$
\end{corollary}

Our second consequence is the following geometric implication of Theorem \ref{thm:Main}, whose proof is postponed until Section \ref{subsec:Proofs}.

\begin{corollary}\label{cor2}
The Poisson limit theorem holds for all curves on the surface with hyperbolic length up to $o(\log\log N)$.
\end{corollary}

From the fact that the genus $g$ of our surfaces is bounded from above linearly in $N$, more precisely by ${N+1\over 2}$, it follows that Corollary \ref{cor2} can be rephrased as follows.

\begin{corollary}\label{cor3}
The Poisson limit theorem holds for all curves on the surface with hyperbolic length up to $o(\log\log g)$.
\end{corollary}

Finally, we remark that because of Lemma \ref{lem:brooks} and Proposition \ref{prp:cusplength}, Corollary \ref{cor2} holds for both the surface with cusps as well as the compactified surface.

\subsection{Proofs}\label{subsec:Proofs}

\begin{proof}[Proof of Theorem \ref{thm:Main}] We first need to set up some more notation. Instead of directly counting occurrences of the words in $W$, we will count occurrences of labelled versions of the words. For $k\in\NN$ let $\mathcal{P}_k $ denote the set of sequences of $k$ ordered pairs of labels in $\{1,\ldots, 6N\}$ such that every pair of labels belongs to the same triangle. Given such a sequence $\alpha\in\mathcal{P}_k$, it naturally gives rise to a word $\mathrm{wrd}(\alpha)$. Every ordered pair of labels in $\alpha$ corresponds to an ordered pair of sides and hence gives us an $L$ or an $R$. Concatenating these letters gives $\mathrm{wrd}(\alpha)$. 

We can also speak of whether or not $\alpha\in\mathcal{P}_k$ can be found in an element $\omega\in\Omega_N$. If this is the case then we will write $\alpha\subset\omega$. Note that $\alpha\subset \omega$ does not imply that the pairs of labels in $\alpha$ form pairs of labels in $\omega$ but rather that the last and first label of every consecutive two pairs in $\alpha$ form a pair in $\omega$.

Now define the sets
\begin{equation*}
\Gamma_{w} := \set{\alpha \in \mathcal{P}_k}{\mathrm{wrd}(\alpha)=w}\quad\text{and}\quad \Gamma = \bigcup_{w\in W}\Gamma_w
\end{equation*}
which will take over the role of the sets $\Gamma_i$ of Section \ref{sec:prob}. For $\alpha\in\Gamma_w$ we define the random variables
\begin{equation*}
Z_{\alpha,N}:\Omega_N\to \{0,1\}
\end{equation*}
by putting
\begin{equation*}
Z_{\alpha,N}(\omega) := \left\{\begin{array}{ll}1 &: \text{if }\alpha\subset\omega \\ 0 &: \text{otherwise}\,. \end{array}\right.
\end{equation*}
Then, we have the representation
\begin{equation}\label{eq:RepresentationZWN}
Z_{[w],N} = \frac{\card{[w]}}{2\len{w}}\sum_{\alpha\in \Gamma_{w}} Z_{\alpha,N}\,.
\end{equation}
The factor $\frac{\card{[w]}}{2\len{w}}$ up front comes from the following consideration. We should consider cyclic permutations of a sequence of pairs $\alpha$ and the sequence read backwards equivalent, because if $\alpha$ and $\alpha'$ differ by these operations then $Z_{\alpha',N}=Z_{\alpha,N}$, as explained in Section \ref{sec:GeometryOfGeodesics}. This is where the factor $\frac{1}{2\len{w}}$ comes from. On the other hand, we do need to consider all the words equivalent to $w$, which is the reason for the factor $\card{[w]}$.

Since we would like to apply the results from Section \ref{sec:prob}, we need to find a decomposition of our set $\Gamma$ and random variables $Z_{\alpha,\beta,N}':\Omega_N\to \{0,1\}$ for every $\alpha\in\Gamma$. Given $\alpha\in\Gamma_{w}$, we define the sets
\begin{equation*}
\Gamma_{\alpha}':=\set{\beta\in\Gamma\setminus\{\alpha\} }{\alpha \text{ and } \beta \text{ have labels in common}}\quad\text{and}\quad\Gamma_\alpha = \Gamma\setminus (\Gamma_\alpha'\cup\{\alpha\})\,.
\end{equation*}
In order to define the random variables $X_{\alpha,\beta}'$, we first define maps $\Omega_N\to\Omega_N$ by $\omega\mapsto \omega_\alpha'$ for every $\alpha\in\Gamma$. Here, $\omega_\alpha'$ is obtained from $\omega$ by ``inserting'' $\alpha$. That is, if $(a_1,a_2)$ is a pair of labels in $\alpha$ and $\omega$ contains the pairs $\{a_1,x\}$ and $\{a_2,y\}$, we remove these pairs from $\omega$ and replace them with $\{a_1,a_2\}$ and $\{x,y\}$. We do this until all pairs from $\alpha$ appear in $\omega$. It is not hard to check that given $\alpha\in\Gamma$, $\omega\mapsto \omega'_\alpha$ is a well-defined constant-to-one map. Now define $X_{\alpha,\beta}':\Omega_N\to\{0,1\}$ by
\[Z_{\alpha,\beta,N}'(\omega) = Z_{\beta,N}(\omega_\alpha') \]
for all $\alpha,\beta\in\Gamma$ and $\omega\in\Omega_N$. Because the map $\omega\mapsto\omega_\alpha'$ is constant-to-one, we have
\[\PP[Z_{\alpha,\beta,N}' = 1] = \PP[Z_{\beta,N}=1 | Z_{\alpha,N} =1].\]
Furthermore, if $\beta$ and $\alpha$ have no labels in common and $\omega$ contains $\beta$, then $\omega_\alpha'$ also contains $\beta$. As such we have that
\[Z_{\alpha,\beta,N}'(\omega) \geq Z_{\beta,N}(\omega) \]
for all $\omega\in\Omega_N$ and all $\beta\in\Gamma_\alpha$. Again we note that it is only the existence of these variables that matters, we will not use them in what follows.

In order to apply Theorem \ref{thm:SteinMulti} let us define the following quantities:
\begin{align*}
\Sigma_{1,w} & := \sum_{\alpha\in\Gamma_{w}}\EE_N[Z_{\alpha,N}]^2\,, & \Sigma_{2,w} & :=\sum_{\alpha\in\Gamma_{w}}\sum_{\beta\in\Gamma_\alpha'}\big(\EE_N[Z_{\alpha,N}]\big)\big(\EE_N[Z_{\beta,N}]\big)
\\
\Sigma_{3,w} & :=\sum_{\alpha\in\Gamma_{w}}\sum_{\beta\in\Gamma_{\alpha}'}\EE_N[Z_{\alpha,N}Z_{\beta,N}]
\end{align*}
and
$$
\Sigma_{4,w}  :=\sum_{\alpha\in\Gamma_{w}}\sum_{\beta\in\Gamma_{\alpha}}\EE_N[Z_{\alpha,N}Z_{\beta,N}]-\big(\EE_N[Z_{\alpha,N}]\big)\big(\EE_N[Z_{\beta,N}]\big)\,.
$$
Note however that because $Z_{[w],N}$ is not just a sum of the $Z_{\alpha,N}$'s, we will need to modify the quantities above slightly to get the actual upper bound on $d_{mTV}$. Nevertheless most of the proof now consists of bounding these three quantities.

We start with $\Sigma_{1,w}$. It will be convenient to decompose the set $\Gamma_{w}$ as the disjoint union
\begin{equation*}
\Gamma_{w}=\Gamma_{w,d}\sqcup\Gamma_{w,nd}\,,
\end{equation*}
where $\Gamma_{w,d}$ is the set of $\alpha\in\Gamma_{w}$ such that all the pairs of labels in $\alpha$ correspond to distinct triangles and $\Gamma_{w,nd}=\Gamma_{w}\setminus\Gamma_{w,d}$.
We will now split $\Sigma_{1,w}$ up according to this decomposition. To shorten notation, we introduce for all $k,N\in\NN$ the following two symbols:
\begin{equation*}
p_{k,N} := \frac{1}{(6N-1) \cdot (6N-3) \cdots (6N-2k+1)}
\end{equation*}
and
\begin{equation*}
a_{k,N} := 3^k\cdot 2N(2N-1)\cdots (2N-k+1)\,.
\end{equation*}
Note that $p_{k,N}$ is the probability that $k$ specific pairs of labels appear in $\omega\in\Omega_N$, while $a_{k,N}$ is the number of sequences of labels that can be obtained by gluing $k$ distinct labelled triangles together in a cycle with a predefined orientation and recording the labels of the identified pairs of sides.

We have
\begin{equation}\label{eq:Sig11}
\sum_{\alpha\in\Gamma_{w,d}}\EE_N[Z_{\alpha,N}]^2 = \sum_{\alpha\in\Gamma_{w,d}} \PP_N[\alpha\subset\omega]^2 = a_{\len{w},N}\;p_{\len{w},N}^2\,,
\end{equation}
which in turn follows from the fact that
\begin{equation*}
\PP_N[\alpha\subset\omega] = p_{\len{w},N} \quad\text{for all}\quad\alpha\in\Gamma_{[w],nd}\quad\text{and}\quad\card{\Gamma_{w,d}} =  a_{\len{w},N}\,.
\end{equation*}
For the other part of the sum in $\Sigma_{1,w}$ we have that
\begin{equation}\label{eq:Sig12}
\sum_{\alpha\in\Gamma_{w,nd}} \EE_N[Z_{\alpha,N}]^2 \leq \sum_{i=1}^{\len{w}-2} 3^i\;(\len{w}-i)^{\len{w}} a_{\len{w}-i,N}\;p_{\len{w}-i,N}^2\,,
\end{equation}
where the number $\len{w}-i$ in the sum should be seen as the number of distinct triangles that are used for the labels (which should be at least $2$, whence the upper bound $|w|-2$ in the above summation). Moreover, $a_{\len{w}-i,N}$ counts the number of ways to pick $\len{w}-i$ triangles and $\len{w}-i$ turns on them, while $(\len{w}-i)^{\len{w}}$ gives an upper bound for the number of ways to use the chosen triangles to form cycles corresponding to $[w]$ (we have to choose $1$ triangle per letter in $w$ and have at most $\len{w}-i$ choices for each letter). The factor $3^i$ bounds the number of ways to choose the $i$ further turns that hadn't been chosen yet.

Adding the terms \eqref{eq:Sig11} and \eqref{eq:Sig12}, we obtain
\begin{equation*}
\Sigma_{1,w} \leq a_{\len{w},N}\;p_{\len{w},N}^2+ \sum_{i=1}^{\len{w}-2} 3^i\;(\len{w}-i)^{\len{w}} a_{\len{w}-i,N}\;p_{\len{w}-i,N}^2\,.
\end{equation*}

For $\Sigma_{2,w}$ we have
\begin{equation*}
\begin{split}
\Sigma_{2,w} \leq \sum_{\substack{w'\text{ s.t. }\\ [w']\in W}}\sum_{i=1}^{2\len{w}} \sum_{j=0}^{\len{w}} \sum_{k=0}^{\len{w'}} \binom{2\len{w}}{i} \, & 3^{i+j+k} (\len{w}-j)^{\len{w}}(\len{w'}-k)^{\len{w'}} \\
& \times a_{\len{w}+\len{w'}-i-j-k,N} \; p_{\len{w},N} \; p_{\len{w'},N}\,.
\end{split}
\end{equation*}
Indeed, $i$ counts how many of these labels $\alpha$ and $\beta\in\Gamma_\alpha'$ share and $j$ and $k$ count the number of repetitions in $\alpha$ and $\beta$ themselves. The binomial coefficient counts the number of ways to choose which of the labels are doubled between $w$ and $w'$. Applying a reasoning similar to the one we have used for $\Sigma_{1,w}$ gives the upper bound. 

For $\Sigma_{3,w}$ we obtain
\begin{equation*}
\begin{split}
\Sigma_{3,w} \leq \sum_{\substack{w'\text{ s.t. }\\ [w']\in W}}\sum_{i=1}^{\len{w}} \sum_{j=0}^{\len{w}} \sum_{k=0}^{\len{w'}} \binom{\len{w}}{i}\, & 3^{i+j+k} (\len{w}-j)^{\len{w}}(\len{w'}-k)^{\len{w'}} \\
& \times a_{\len{w}+\len{w'}-i-j-k-1,N} \; p_{\len{w}+\len{w'}-i,N} \,.
\end{split}
\end{equation*}
This is based on the following observation. The only way for $Z_{\alpha,N}$ and $Z_{\beta,N}$ to concurrently be equal to $1$ is when the pairs in $\beta$ do not violate those in $\alpha$. This means that given $\alpha$, the $\beta$'s that add a non-zero contribution to $\Sigma_{3,w}$ together with $\alpha$ can be formed by picking some number of pairs where they intersect. The index $i$ counts how many pairs $\alpha$ and $\beta$ intersect in and the binomial coefficient counts the number of ways to pick these pairs in $\alpha$. The extra `$-1$' in the index of $a$ comes from the fact that if two circuits intersect in $i$ pairs, they intersect in at least $i+1$ vertices.

Now, we will use our assumptions on $m_W$ to uniformly bound the three terms $\Sigma_{1,w}$, $\Sigma_{2,w}$ and $\Sigma_{3,w}$. The other crucial ingredient in this bound will be that for all $k,l,N\in\NN$ such that $k\leq l\leq N$ we have
\begin{equation}\label{eq:Bd1}
a_{k,N}p_{l,N} \leq \left(1+\frac{1}{6N-1}\right) N^{k-l} \leq \frac{6}{5} N^{k-l}\,.
\end{equation}
This is proved as follows. The expressions for $a_{k,N}$ and $p_{l,N}$ contain $k$ and $l$ factors with an $N$ in them respectively. The factors in $p_{l,N}$ are roughly $3$ times those in $a_{k,N}$. Comparing these factors one by one, one sees that the first factor in $a_{k,N}$ is slightly more than $1/3$ of the corresponding factor in $p_{l,N}$, hence the $(1+\frac{1}{6N-1})$. All the other factors in a are less than $1/3$ of the corresponding one in $p_{l,N}$. So in particular, if $l \geq k$, they kill off the $3^k$ in $a_{k,N}$.  The remaining factors in $p_{l,N}$ make up the $N^{k-l}$ (we silently assume here that the factors in $p_{l,N}$ are all more than $N$, which comes down to $l<5N/2$, which in all the applications later in the text is satisfied because of the bound on $m_W$). Similarly, one verifies that
\begin{equation}\label{eq:Bd2}
a_{k,N}\,p_{l,N}\,p_{m,N} \leq {6\over 5}\,N^{k-l-m}
\end{equation}
for all $k,l,m\leq N$.

We define $c_W:=\max\set{\card{[w]}}{[w]\in W}$ and compute, using \eqref{eq:Bd1} and \eqref{eq:Bd2}, 
\begin{align*}
\Sigma_{1,w} &\leq {6\over 5}{1\over N}+\sum_{i=1}^{\len{w}-2}3^i(\len{w}-i)^{\len{w}}\,{6\over 5}{1\over N^{\len{w}-i}}\\
&\leq {6\over 5}{1\over N}\big(1+(3m_W)^{m_W+1}\big)\,,
\end{align*}
and, similarly,
\begin{align*}
\Sigma_{2,w} \leq {6\over 5}\card{W}c_W(6m_W)^{3m_W+3}{1\over N}\qquad\text{and}\qquad\Sigma_{3,w}\leq {6\over 5}\card{W}c_W(3m_W)^{3m_W+3}{1\over N}\,.
\end{align*}

Finally, for $\Sigma_{4,w}$, we note that a direct computation gives that
\begin{eqnarray*}
p_{k+l,N} - p_{k,N}\cdot p_{l,N} & \leq & p_{k,N}\cdot p_{l,N}\cdot \left(\left(\frac{6N-1}{6N-2(k+l)-1}\right)^{k+l}-1\right) \\[3mm]
& \leq &  \frac{2(k+l)^2}{6N-2(k+l)+1} \, p_{k,N}\cdot p_{l,N}\,.
\end{eqnarray*}
Using our assumption on $m_W$ and a similar bound to the bound on $\Sigma_{1,w}$, we obtain that
\begin{eqnarray*}
\Sigma_{4,w} & \leq & \frac{m_W^2}{N}\left(\sum_{w\in W} a_{\abs{w},N}p_{\abs{w},N}+ \sum_{i=1}^{\len{w}-2} 3^i\;(\len{w}-i)^{\len{w}} a_{\len{w}-i,N}\;p_{\len{w}-i,N}\right)^2 \\
& \leq & \frac{36 (\card{W}\, m_W+(3m_W)^{m_W})^2}{25N},
\end{eqnarray*}
by \eqref{eq:Bd2}.

To finish the proof, we remind the reader of the fact that we still need to modify the quantities above because of the multiplicative factor in front of the variables $Z_{[w],N}$ in terms of the $Z_{\alpha,N}$, recall \eqref{eq:RepresentationZWN}. This modification consists of two things. For every word $w$, we have overcounted: if $\alpha'$ is a cyclic permutation of $\alpha$ or $\alpha$ read backwards, then $Z_{\alpha',N}=Z_{\alpha,N}$, so these `doubles' can be thrown out of the sums above. On the other hand, given a class of words $[w]\in W$, we need to add the sums over all representatives of $[w]$. This procedure gives rise to four quantities $\Sigma_{1,[w]}$, $\Sigma_{2,[w]}$, $\Sigma_{3,[w]}$ and $\Sigma_{4,[w]}$, which form the actual upper bounds for $d_{mTV}$. Formally, we have
\begin{equation*}
\Sigma_{1,[w]} = \frac{\card{[w]}}{2\len{w}}\Sigma_{1,w}\,,\qquad\Sigma_{2,[w]}\leq\left(\frac{\card{[w]}}{2\len{w}}\right)^2\Sigma_{2,w}\,,\qquad\Sigma_{3,[w]}\leq \left(\frac{\card{[w]}}{2\len{w}}\right)^2\Sigma_{3,w}
\end{equation*}
and
$$
\Sigma_{4,[w]}\leq \left(\frac{\card{[w]}}{2\len{w}}\right)^2\Sigma_{4,w}
$$
for all $[w]\in W$. Theorem \ref{thm:SteinMulti} now tells us that
\begin{equation*}
d_{mTV}(\bZ_{W,N},\bZ_W) \leq 3\sum_{[w]\in W}\big(\Sigma_{1,[w]}+\Sigma_{2,[w]} + \Sigma_{3,[w]}+\Sigma_{4,[w]}\big) \leq 3\,\frac{6\card{W}^3(6m_W)^{3m_W+4}}{N}\,,
\end{equation*}
where we have used ${6\over 5}+{6\over 5}+{6\over 5}+{36\over 25}\leq 6$ and the fact that $c_W\leq 2m_W$. Indeed, all elements in the class of word $w$ are obtained by cylclically permuting the letters of $w$ or by reversing the word, which means that one has at most $2\len{w}$ elements in $[w]$. This completes the proof.
\end{proof}

\begin{proof}[Proof of Corollary \ref{cor2}]
Let us introduce for $k\in\NN$ the set $W(k)$ of all (classes of) words that have trace less than or equal to $k$.
Then, it follows from \cite[Proposition 3.4]{PetriWalker} that $\card{W(k)}\leq c_\varepsilon k^{2+\varepsilon}$ for any $\varepsilon>0$ and a constant $c_\varepsilon \in(0,\infty)$. Moreover, in this situation one has that $m_{W(k)}=k-1$. Thus, Corollary \ref{cor1} delivers the bound
\begin{equation}\label{eq:UB}
18\card{W(k)}^3\,(6m_{W(k)})^{3m_{W(k)}+4}\,{1\over N} \leq c_\varepsilon'\,{(6k)^{3k+5+3\varepsilon}\over N}
\end{equation}
on the multivariate total variation distance for some constant $c_\varepsilon'\in(0,\infty)$ that only depends on $\varepsilon$.

Now, let $\phi_\varepsilon(x)$ be the inverse of the function $x\mapsto c_\varepsilon'(6x)^{3x+5+3\varepsilon}$, $x\in(0,\infty)$. It is easily verified that $\phi_\varepsilon$ satisfies the estimate
$$
(\log x)^\alpha\leq \phi_\varepsilon(x)\leq\log x
$$
for all $x\geq 1$ and $\alpha\in(0,1)$. In particular, if $k=o((\log N)^\alpha)$, then the upper bound \eqref{eq:UB} on the multivariate total variation distance tends to zero and hence the Poisson limit theorem holds. This can be rephrased by saying that the length $\ell(\gamma)$ of a corresponding curve $\gamma$ has to satisfy
$$
\ell(\gamma) =o\big(\cosh^{-1}\big({(\log N)^\alpha}\big)\big) = o(\log\log N)\,,
$$
since $\cosh^{-1}$ behaves roughly like a logarithm for large arguments.
\end{proof}


\end{document}